\documentclass[conference]{IEEEtran}
\IEEEoverridecommandlockouts
\pdfoutput=1

\usepackage{graphicx}
\usepackage{graphics}
\usepackage{svg}
\usepackage{float}
\usepackage{comment}
\usepackage{xcolor}
\usepackage{url}
\usepackage[normalem]{ulem}
\usepackage{soul}
\usepackage{setspace}
\usepackage{subcaption}

\usepackage{amsmath}
\usepackage{amssymb}
\usepackage{amsthm}
\usepackage{mathtools}
\usepackage{mathabx}

\usepackage{algorithm}
\usepackage{algpseudocode}

\usepackage{cite}
\usepackage[hang,flushmargin]{footmisc}
\usepackage[caption=false,font=footnotesize]{subfig}

\usepackage{tikz}
\usepackage{pgfplots}
\usepackage{pgfplotstable}
\usepackage{graphviz}
\usetikzlibrary{math}
\usetikzlibrary{shapes,arrows}
\usetikzlibrary{arrows.meta}
\usetikzlibrary{positioning}
\tikzset{>={Latex[width=2.5mm,length=2.5mm]}}
\tikzstyle{block}=[draw opacity=0.7,line width=1.4cm]
\pgfplotsset{compat=1.18}
\usetikzlibrary{patterns}
\usepackage{subcaption}
\usetikzlibrary{arrows.meta,positioning,fit,calc,backgrounds,decorations.pathreplacing}

\newcommand{\boxend}{\hfill \ensuremath{\Box}}
\newcommand{\longthmtitle}[1]{\mbox{}\textit{{(#1):}}}

\newtheorem{thm}{Theorem} 
\newtheorem{rem}{Remark} 
\newtheorem{cor}{Corollary} 
\newtheorem{lem}{Lemma} 
\newtheorem{defn}{Definition}

\newcommand{\IF}{\mathcal{I}}

\begin{document}

\title{\Large{The Price of Feasibility: Greedy Approximation Bounds for
String Supermodular Optimization over Oracle-Conditioned Greedoids}}

\author{Joan Vendrell$^{\dagger}$, Russell Bent$^{\ddagger}$,
and Solmaz Kia$^{\dagger}$, \textit{Senior Member, IEEE} 
\thanks{$^\dagger$ First and third authors are with the Mechanical and Aerospace Engineering Department of University of California Irvine, Irvine, CA, USA
({\tt\small jvendrel,solmaz@uci.edu}). Their work was supported by NSF Award ECCS 2452149. $^{\ddagger}$The second author is with the Los Alamos National Laboratory, NM, USA
    }
}

\maketitle

\begin{abstract}
Greedy algorithms efficiently approximate combinatorial optimization problems, but their guarantees weaken when feasibility couples combinatorial structure with global physical constraints. We study monotone nondecreasing supermodular minimization over the bases of a graphic greedoid under physics-induced constraints. We model physics-informed selection using a look-ahead oracle that identifies candidates extendable to a feasible basis, yielding the \emph{Conditioned Sequential Greedy Algorithm}. We derive a closed-form approximation bound, its \emph{price of feasibility}, based on the variability of oracle-restricted candidate sets and a probabilistic correction for unobserved elements. As a case study, we show that \texttt{FORWARD}, an algorithm for multi-source radial network reconfiguration, instantiates this framework, and numerically demonstrate the bound tightness and the feasibility–optimality trade-off quantification.
\end{abstract}



\section{Introduction}
\label{sec:intro}

Combinatorial optimization is central to dynamical systems and network design, spanning sensor deployment, scheduling, resource allocation, and network topology design~\cite{CP-KS:98,BG-MM:04,AK-DG:14bookChap,SSK:26-ency}. Structured problems such as sensor placement and scheduling are often formulated as submodular maximization under cardinality or matroid constraints~\cite{VT-AJ-GJP:16,JQ-IY-RR:19,RL-NM-RH:23}. The Sequential Greedy Algorithm (\texttt{SGA}) then achieves approximation ratios of $(1-1/e)$ and $1/2$, respectively~\cite{GN-LW-MF:78,MLF-GLN-LAW:78}. These guarantees, however, require objective evaluation on arbitrary partial solutions.

Indeed, in many cyber-physical systems, greedy selection based solely on local cost is inadmissible because objective values are well-defined only under global structural constraints. For instance, network flow and routing problems require connectivity to one or more sources for flows and thus costs to be evaluated. In the single-source setting, Khodabakhsh et al.~\cite{AK-GY-SB-EV-MC-TL-EP:18} addressed this issue by formulating radial network reconfiguration as supermodular minimization over a graphic matroid. Their method performs branch exchanges from an initially feasible spanning tree (\emph{basis}), preserving flow balance at every iteration. This decouples physical feasibility from greedy selection, enabling standard matroid analysis and the first approximation guarantees for single-source radial reconfiguration.

However, this approach relies on a pre-defined feasible basis to initialize the exchange, effectively bypassing the challenge of ensuring a source-to-sink path is present during the growth of the network. Hence, this formulation does not extend when a solution must be constructed from scratch, and when the existence of a flow-consistent path must be actively certified at every step of the greedy expansion.

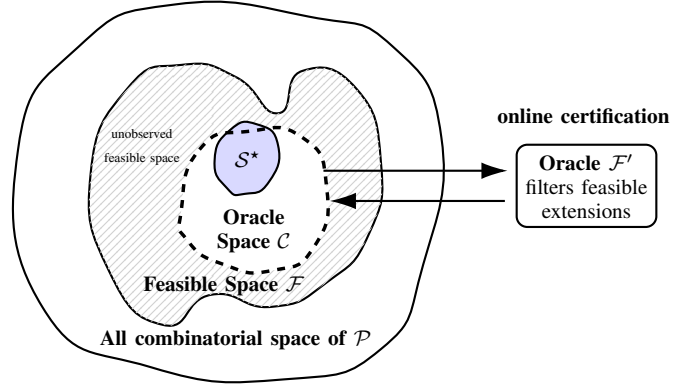
\begin{figure}[t]
    \centering
    \begin{tikzpicture}[scale=0.8, every node/.style={font=\footnotesize}]
    
    \draw[thick]
      plot[smooth cycle, tension=0.9] coordinates {
        (-3.2,-1.8) (-3.6,0.8) (-2.6,2.5) (-1.0,3.2)
        (1.2,3.0) (2.6,2.5) (3.4,1.0) (3.1,-1.1)
        (2.4,-2.5) (0.8,-3.0) (-1.4,-2.8)
      };
    \node at (0,-2.35) {\textbf{All combinatorial space of $\mathcal{P}$}};
    
    \draw[thick]
      plot[smooth cycle, tension=0.95] coordinates {
        (-2.1,-1.2) (-2.4,0.4) (-1.8,1.8) (-0.9,2.2)
        (0.0,2.0) (0.7,1.4) (1.1,2.1) (2.0,1.8)
        (2.4,0.6) (2.2,-0.6) (1.5,-1.4) (0.6,-1.8)
        (-0.4,-1.6) (-1.2,-2.1)
      };

    \begin{scope}
      \clip plot[smooth cycle, tension=0.95] coordinates {
        (-2.1,-1.2) (-2.4,0.4) (-1.8,1.8) (-0.9,2.2)
        (0.0,2.0) (0.7,1.4) (1.1,2.1) (2.0,1.8)
        (2.4,0.6) (2.2,-0.6) (1.5,-1.4) (0.6,-1.8)
        (-0.4,-1.6) (-1.2,-2.1)
      };
      \fill[pattern=north east lines, pattern color=gray!35]
        (-3.5,-3.2) rectangle (3.5,3.2);
    \end{scope}
    
    \begin{scope}
      \clip plot[smooth cycle, tension=1.0] coordinates {
        (-0.9,-0.2) (-0.6,0.8) (0.2,1.1) (1.1,1.0)
        (1.5,0.2) (1.3,-0.8) (0.6,-1.2) (-0.2,-1.1)
        (-0.8,-0.7)
      };

      \fill[white] (-2,-2) rectangle (2,2);
    \end{scope}

    \begin{scope}[yshift=0.5]
    \draw[thick, fill=blue!15]
      plot[smooth cycle, tension=1.0] coordinates {
    (-0.15,0.15) (-0.35,0.60) (-0.1,1.10) (0.35,1.20)
    (0.7,0.85) (0.55,0.25) (0.15,0.05)};
    \node at (0.2,0.62) {\(\mathcal{S}^\star\)};
    \end{scope}
    
    \draw[dashed, line width=1.3pt]
      plot[smooth cycle, tension=1.0] coordinates {
        (-0.9,-0.2) (-0.6,0.8) (0.2,1.1) (1.1,1.0)
        (1.5,0.2) (1.3,-0.8) (0.6,-1.2) (-0.2,-1.1)
        (-0.8,-0.7)
      };
    \node[align=center] at (0.3,-0.55) {\footnotesize \textbf{Oracle} \\ \footnotesize \textbf{Space} $\mathcal{C}$};
    
    \node[align=center] at (-1.55,0.85) {\tiny unobserved \\ \tiny feasible space};
    
    
    
    
    
    \node[
      draw,
      rounded corners=4pt,
      thick,
      align=center,
      minimum width=0.8cm,
      minimum height=0.8cm
    ] (oraclebox) at (5.8,0.2) {\textbf{Oracle \(\mathcal{F}^\prime\)}\\filters feasible\\extensions};
    
    \node at (5.75,1.35) {\textbf{online certification}};
    
    \draw[-{Latex[length=3.5mm,width=2.2mm]}, thick] (1.45,0.45) -- (4.5,0.45);
    \draw[-{Latex[length=3.5mm,width=2.2mm]}, thick] (4.45,-0.05) -- (1.55,-0.05);
    
    \node at (-0.2,-1.45) {\textbf{Feasible Space} $\mathcal{F}$};
    \end{tikzpicture}
\caption{Oracle-restricted exploration of an approximated feasible region. The hatched area denotes feasible solutions outside the oracle-restricted search space. Note that the optimal space can yield outside the oracle-restricted space.}
\label{fig:oracle_irregular}
\end{figure}

In a more general setting, structural requirements can be enforced by moving from matroid search spaces to the richer class of \emph{greedoids}~\cite{BK-LL:81,JY:85}. In greedoids, the hereditary property of matroids is specialized to an accessibility requirement, where feasibility becomes order-dependent. While recent theoretical work has explored the polymatroid representation of greedoids to justify the success of greedy exchange properties~\cite{RS-VKG:26}, and others have established optimality guarantees for \emph{string submodular} functions over ordered structures~\cite{ZZ-EKPC-AP-WM-SDH:12,ZZ-EKPC-AP-WM:16}, these works typically assume that structural accessibility alone is sufficient for feasibility. However, in the cyber-physical problems of interest here, even a solution that respects the greedoid's structural accessibility must also be completable to a basis satisfying continuous constraints imposed by the underlying physics. Since these physical constraints are often only verifiable globally at the basis level, the greedy search must prune the candidate space beyond the structural conditioning of the greedoid alone. This necessitates a feasibility oracle that actively filters the greedoid expansion to prevent termination in physically inadmissible states.

This simultaneous conditioning on both structural and physical feasibility introduces a fundamental and largely unresolved tension. On one hand, physics-induced infeasibility at each iteration of a greedy algorithm may restrict the candidate set non-uniformly, pruning branches of the search space that might otherwise contain high-quality solutions. On the other hand, feasibility-preserving decisions may sacrifice local optimality to ensure long-term consistency, and the resulting gap between the conditioned sequential greedy solution and the global optimum---what we term the \emph{price of feasibility}---remains, in general, uncharacterized, see Fig.~\ref{fig:oracle_irregular}. 

More precisely: \emph{how can one expand greedy algorithms over restricted combinatorial spaces while guaranteeing termination in a physically feasible basis?} Establishing rigorous optimality guarantees for such \emph{conditioned sequential greedy algorithms} is an open theoretical problem with direct implications for infrastructure network design and distributed control, where exact optimization is computationally intractable. This paper investigates this problem by introducing a formal framework for conditioned greedy optimization and deriving novel approximation bounds for this class of algorithms.

\paragraph*{Contributions}
We introduce a framework for minimizing a monotone nondecreasing supermodular function\footnote{For $C\geq\max_{\mathcal{S}\in\mathcal{B}(\mathcal{I})}f(\mathcal{S})$, the function $C-f$ is nonnegative and submodular, with $\arg\min f=\arg\max(C-f)$; thus, greedy minimization selects the smallest rather than the largest marginal value~\cite{GN-LW-MF:78}.} over graphic-greedoid bases subject to global physical equality constraints. Physics-informed candidate selection is modeled by a look-ahead feasibility oracle $\mathcal{F}$ that retains only elements extendable to a physically feasible basis. This formulation yields the \emph{Conditioned Sequential Greedy Algorithm} and reveals how incomplete oracle coverage affects optimality.

We derive a closed-form approximation bound governed by the variability of oracle-restricted candidate sets, with a probabilistic correction for unobserved elements that quantifies the \emph{price of feasibility}. We then show that \texttt{FORWARD} is a canonical instance for multi-source radial network reconfiguration, a problem commonly addressed using MINLP or heuristics without guarantees~\cite{JAT:12,NS:19,MB-FW:89,JV-RB-SK:25arxiv}. Applying our bound establishes the first formal optimality guarantees for \texttt{FORWARD}, while experiments on standard IEEE feeders demonstrate its tightness.

\paragraph*{Organization}
Section~\ref{sec:preliminaries} introduces the preliminary definitions. Section~\ref{sec:problem} formalizes the abstract problem, presents the Conditioned Sequential Greedy Algorithm, and defines the role of the look-ahead feasibility oracle. Section~\ref{sec:algorithm} proves the main optimality gap theorem. Section~\ref{sec:forward} instantiates the framework for \texttt{FORWARD} as a canonical case study and validates the theoretical predictions experimentally. Section~\ref{sec:conclusion} concludes.

\section{Preliminaries}
\label{sec:preliminaries}
This section collects the combinatorial and functional definitions used throughout the paper. We introduce the greedoid framework as the foundational structure for the feasible set of the optimization problem, and define supermodularity restricted to a greedoid domain, which is the natural setting for problems where the objective function is only defined over combinatorially feasible configurations.

In classical submodular optimization, a configuration is typically viewed as a set $\mathcal{S} \subseteq \mathcal{P}$, where the order of elements is irrelevant. However, in many cyber-physical systems, the "state" of the system depends on the sequence in which elements are added. 

Let $\mathcal{P}$ be a finite ground set. We denote by $\mathcal{P}^*$ the set of all finite \emph{strings} (sequences) of distinct elements from $\mathcal{P}$. For any two strings $\alpha, \beta \in \mathcal{P}^*$, we distinguish them not only by their content but by their ordering. Specifically, we consider functions $f: \mathcal{P}^* \to \mathbb{R}$ where:
\begin{equation}
    f((s_1, s_2, s_3)) \neq f((s_2, s_1, s_3)).
\end{equation}
This distinction is vital for power systems, where the marginal cost of adding a line $e$ depends on the specific topology already established by the sequence $\mathcal{S}$.

\subsection*{Combinatorial Structure}

\begin{defn}[Greedoid~\cite{BK-LL:81}]
\label{def:greedoid}
Let $\mathcal{P}$ be a finite ground set. A pair
$(\mathcal{P}, \IF)$ is a \emph{greedoid} if:
\begin{itemize}
    \item[\textbf{(G1)}] $\emptyset \in \IF$.
    \item[\textbf{(G2)}] \textbf{Accessibility:} For every
        nonempty $\mathcal{S} \in \IF$, $\exists\, s \in
        \mathcal{S}$ such that
        $\mathcal{S} \setminus \{s\} \in \IF$.
    \item[\textbf{(G3)}] \textbf{Asymmetric Exchange:} For
        any $\mathcal{S}_1, \mathcal{S}_2 \in \IF$ with
        $|\mathcal{S}_1| < |\mathcal{S}_2|$, $\exists\,
        s_2 \in \mathcal{S}_2 \setminus \mathcal{S}_1$
        such that
        $\mathcal{S}_1 \cup \{s_2\} \in \IF$.
\end{itemize}
\boxend
\end{defn}

\begin{defn}[Basis of a Greedoid~\cite{BK-LL:81}]
\label{def:basis}
A \emph{basis} of a greedoid $(\mathcal{P}, \IF)$ is a maximal feasible set $\mathcal{S} \in \IF$, i.e., there exists no $e \in \mathcal{P} \setminus \mathcal{S}$ such that $\mathcal{S} \cup \{e\} \in \IF$. The collection of all bases is denoted $\mathcal{B}(\mathcal{I})$.  \boxend
\end{defn}

\begin{rem}[]
    Asymmetric exchange property of greedoids allows us to capture non-commutative nature of the process distinguishing between $f((i,j))\neq f((j,i))$, unlike symmetry restricted by matroids, allowing a unique mapping between $f(\mathcal{S})\to\mathbb{R}$.
    \label{rem:importance_greedoid}
\end{rem}

 Elements of a greedoid are sequences, i.e., $\mathcal{S}=(s_1,s_2,s_3)\neq (s_2,s_3,s_1)$. Unlike matroids, greedoid bases need not share the same cardinality in general; however, for the graphic greedoids over a graph of $N$ nodes arising from rooted spanning forest constructions considered in this work, all bases have common cardinality $N-1$.

\begin{defn}[$\mathcal{R}$-Rooted Spanning Forest Greedoid]
\label{def:polyforest_greedoid}
Let $\mathcal{G} = (\mathcal{V}, \mathcal{P})$ be a directed graph with designated root set $\mathcal{R} \subseteq \mathcal{V}$. The feasible collection
\begin{align*}
    \mathcal{I} = \bigl\{&\mathcal{S} \subseteq \mathcal{P} \;:\; \mathcal{S} \text{ is acyclic, every node has in-degree at}\\ &\text{most one, and every component of } (\mathcal{V}, \mathcal{S}) \text{ contains at}\\&\text{least one node in } \mathcal{R} \bigr\}
\end{align*}
defines a greedoid $(\mathcal{P}, \mathcal{I})$. A basis $\mathcal{S} \in \mathcal{B}(\mathcal{I})$ is a spanning forest of $\mathcal{G}$ in which every tree contains at least one node from $\mathcal{R}$. Nodes in $\mathcal{R}$ are not required to have in-degree zero: a root node may receive from another root node when it acts as an intermediary connecting downstream components beyond its individual capacity. \boxend
\end{defn}

The greedoid $(\mathcal{P}, \mathcal{I})$ is a generalization of the directed branching greedoid~\cite{BK-LL:81} to a multi-root setting where nodes in $\mathcal{R}$ are not required to have in-degree zero, accommodating configurations in which a root node receives from another root node to relay supply to downstream components.

\subsection*{Function Structure}
\begin{defn}[Normalized Monotone Non-Decreasing Function over a Greedoid]
\label{def:monotone}
Over a greedoid $(\mathcal{P}, \mathcal{I})$, a set function $f : \mathcal{I} \to \mathbb{R}_{\geq 0}$ is \emph{normalized} if $f(\emptyset) = 0$ and \emph{monotone non-decreasing} if for all $\mathcal{S}_1, \mathcal{S}_2 \in \mathcal{I}$ with $\mathcal{S}_1 \subseteq \mathcal{S}_2$, $f(\mathcal{S}_1) \leq f(\mathcal{S}_2)$. \boxend
\end{defn}

\begin{defn}[Supermodular Function over a Greedoid]
\label{def:supermodular}
Let $(\mathcal{P}, \mathcal{I})$ be a greedoid. A set function $f : \mathcal{I} \to \mathbb{R}_{\geq 0}$ is \emph{supermodular over} $\mathcal{I}$ if for all $\mathcal{S}_1 \subseteq \mathcal{S}_2 \in \mathcal{I}$ and all $e \in \mathcal{P} \setminus \mathcal{S}_2$ such that both $\mathcal{S}_1 \cup \{e\}$ and $\mathcal{S}_2 \cup \{e\}$ remain in $\mathcal{I}$,
\[
    \Delta(e \mid \mathcal{S}_2) \;\geq\;
    \Delta(e \mid \mathcal{S}_1),
\]
where $\Delta(e \mid \mathcal{S}) := f(\mathcal{S} \cup \{e\}) - f(\mathcal{S})$ is the \emph{marginal cost} of adding $e$ to $\mathcal{S}$.\footnote{For a \emph{normalized} and monotone non-decreasing supermodular function, $\Delta(e \mid \mathcal{S}) \geq 0$.} \boxend
\end{defn}

Note that while classical theory defines supermodularity over the full power set $2^{\mathcal{P}}$, restricting the domain to $\mathcal{I}$ is not a relaxation---it is a necessity in problems where $f(\mathcal{S})$ depends on a continuous variable uniquely determined by $\mathcal{S}$, making $f$ simply undefined for sets outside the greedoid. Since greedy searches operate exclusively within $\mathcal{I}$, requiring the property to hold elsewhere is both impossible and unnecessary. This formulation therefore broadens the framework to a wider class of cyber-physical problems than the classical setting admits.

\section{Abstract Problem Setting}
\label{sec:problem}

We consider a mixed combinatorial and continuous optimization problem, where the objective couples a discrete set variable $\mathcal{S} \subseteq \mathcal{P}$ over a finite ground set $\mathcal{P}$ of candidate elements with a continuous variable $\mathbf{x}(\mathcal{S}) \in \mathbb{R}^{|\mathcal{S}|}$ that is uniquely determined by $\mathcal{S}$ through a system of equality constraints. Let $\mathcal{I}$ denote a greedoid over $\mathcal{P}$ with basis collection $\mathcal{B}(\mathcal{I})$, and let $g : \mathbb{R}^{|\mathcal{S}|} \to \mathbb{R}^k$ encode physical or operational equality constraints that must hold at every complete configuration. The abstract problem reads:

\begin{subequations}\label{eq:abstract}
\begin{align}
    \underset{\mathcal{S}}{\text{minimize}}
    \quad
    & f\!\left(\mathcal{S},\mathbf{x}(\mathcal{S})\right)
    \label{eq:abstract_obj}\\
    \text{subject to} \quad
    & \mathcal{S} \in \mathcal{B}(\mathcal{I}),
    \label{eq:abstract_greedoid}\\
    & g\!\left(\mathbf{x}(\mathcal{S})\right) = 0.
    \label{eq:abstract_feasibility}
\end{align}
\end{subequations}
Here $\mathbf{x}(\mathcal{S})$ denotes the restriction of a global continuous variable to the components activated by $\mathcal{S}$; its ambient dimension is immaterial to the analysis. Since $\mathbf{x}(\mathcal{S})$ is uniquely determined by $\mathcal{S}$ through~\eqref{eq:abstract_feasibility}, we write $f(\mathcal{S})$ as shorthand for $f(\mathcal{S},\mathbf{x}(\mathcal{S}))$ and treat the problem as a pure set optimization over $\mathcal{B}(\mathcal{I})$. The greedoid constraint $\mathcal{S} \in \mathcal{B}(\mathcal{I})$ encodes structural requirements that are path-dependent: in radial distribution network reconfiguration, every load node must trace a directed path to at least one of the sources in source set nodes $\mathcal{V}_g$; in multi-terminal network formation, every agent must maintain a path to a designated gateway. In both cases, removing an edge from a valid configuration may disconnect a node from all roots, which is precisely why a matroid is insufficient and a greedoid is the correct domain, as established in Section~\ref{sec:preliminaries}. Optimization is restricted to $\mathcal{B}(\mathcal{I})$ rather than arbitrary members of $\mathcal{I}$ because $f(\mathcal{S})$ is meaningful only when $\mathcal{S}$ is a complete spanning configuration: $\mathbf{x}(\mathcal{S})$ encodes global quantities, flows, potentials, or allocations, that are uniquely determined and physically well-posed only when every node is served.

Minimizing a supermodular function over the bases of a greedoid is NP-hard in general~\cite{CP-KS:98}. Consequently, Problem~\eqref{eq:abstract} is at least NP-hard, motivating tractable approximation algorithms with provable guarantees.

\begin{algorithm}[t]
\caption{Conditioned Sequential Greedy Algorithm}
\label{alg:conditioned_greedy}
\begin{algorithmic}[1]
\Require Ground set $\mathcal{P}$, cost function $f$,
         greedoid $\IF$,
         look-ahead oracle $\mathcal{F}^\prime$
\Ensure $\bar{\mathcal{S}} \in \mathcal{B}(\mathcal{I})$
        with $g(\mathbf{x}(\bar{\mathcal{S}})) = 0$
\State $\bar{\mathcal{S}} \leftarrow \mathcal{S}^\star_0$ \Comment{Where $\mathcal{S}^\star_0\subseteq\mathcal{S}^\star$}
\While{$\bar{\mathcal{S}} \notin \mathcal{B}(\mathcal{I})$}
    \State $\mathcal{C} \leftarrow
    \bigl\{e \in \mathcal{P} \setminus \bar{\mathcal{S}}
    \;\big|\;
    \bar{\mathcal{S}} \cup \{e\} \in \IF
    \ \text{and}\ 
    \mathcal{F}^\prime(\bar{\mathcal{S}} \cup \{e\}) = 1
    \bigr\}$
    \Comment{Greedoid membership and look-ahead
    physical feasibility}
    \State $\bar{s} \leftarrow
           \arg\min_{s \in \mathcal{C}}\;
           \Delta(s \mid \bar{\mathcal{S}})$
    \Comment{Minimum marginal cost}
    \State $\bar{\mathcal{S}} \leftarrow
           \bar{\mathcal{S}} \cup \{\bar{s}\}$
\EndWhile
\end{algorithmic}
\end{algorithm}
To this end, we propose the \emph{Conditioned Sequential Greedy Algorithm} (Algorithm~\ref{alg:conditioned_greedy}), which builds a solution incrementally by selecting at each step the element of minimum marginal cost from a conditioned candidate set. The candidate set at each step is restricted to elements that are not only greedoid-feasible --- that is, their addition keeps the partial solution within $\mathcal{I}$ --- but also physically completable to a basis satisfying $g(\mathbf{x}(\mathcal{S}')) = 0$. This second condition is enforced by a \emph{look-ahead physical feasibility oracle} $\mathcal{F}$, which screens candidates before selection and may be implemented through prioritization rules, reachability checks, or other look-ahead measures. We defer the concrete realization of $\mathcal{F}$ to Section~\ref{sec:forward}, where the \texttt{FORWARD} Algorithm is introduced as a specific instantiation of this framework.

\begin{defn}[Look-Ahead Physical Feasibility Oracle]
\label{def:oracle}
A \emph{look-ahead physical feasibility oracle} $\mathcal{F} : \mathcal{I} \to \{0,1\}$ returns $\mathcal{F}(\mathcal{S}) = 1$ if and only if there exists a basis $\mathcal{S}' \in \mathcal{B}(\mathcal{I})$ with $\mathcal{S} \subseteq \mathcal{S}'$ such that $g(\mathbf{x}(\mathcal{S}')) = 0$. When $\mathcal{S} \in \mathcal{B}(\mathcal{I})$, this reduces to $\mathcal{F}(\mathcal{S}) = 1$ iff $g(\mathbf{x}(\mathcal{S})) = 0$. \boxend
\end{defn}

The oracle is distinct from greedoid membership: structural requirements such as acyclicity and rooted connectivity are enforced by $\mathcal{I}$ through local edge selection. The oracle addresses an orthogonal question, whether a greedoid-feasible partial solution can still be completed to a basis satisfying the global physical constraint $g(\mathbf{x}) = 0$, and is necessary precisely because this property cannot be verified locally or incrementally from the partial solution alone. Termination of Algorithm~\ref{alg:conditioned_greedy} is implicit in Definition~\ref{def:oracle}: if at some iteration $\mathcal{F}(\bar{\mathcal{S}}\!\cup\!\{e\}) \!\!=\!\! 1$ was satisfied, then $\bar{\mathcal{S}}$ is already contained in some physically feasible basis $\mathcal{S}' \in \mathcal{B}(\mathcal{I})$, and the greedoid asymmetric exchange property~(G3) guarantees the existence of at least one augmenting element $e' \in \mathcal{S}' \setminus \bar{\mathcal{S}}$ with $\mathcal{F}(\bar{\mathcal{S}} \cup \{e'\}) \!\!=\!\! 1$, keeping $\mathcal{C}$ non-empty until a basis is reached.

\section{Optimality Gap of the Conditioned Sequential Greedy Algorithm}
\label{sec:algorithm}
In this section we analyze the optimality gap of Algorithm~\ref{alg:conditioned_greedy}. The key element is the structure of the oracle space $\mathcal{C}_t$ at iteration $t$, whose definition determines which elements are ever considered for selection and therefore governs both the quality of the solution and the need for a probabilistic correction. Classical sequential greedy for  minimization over a uniform matroid  selects at each step the element $e$ with the smallest marginal cost from $\mathcal{P} \setminus \bar{\mathcal{S}}$. In Problem~\eqref{eq:abstract}, unrestricted expansion may violate either the greedoid structure or the prospect of satisfying the physical feasibility constraint~\eqref{eq:abstract_feasibility}. This is resolved by gating candidates through (line 3)
\[
    \mathcal{C}_t = \bigl\{e \in \mathcal{P} \setminus
    \mathcal{S}_{t-1} \;:\; \mathcal{S}_{t-1} \cup \{e\}
    \in \mathcal{I}\ \text{and}\
    \mathcal{F}(\mathcal{S}_{t-1} \cup \{e\}) = 1\bigr\},
\]
where the first condition enforces greedoid membership and the second enforces look-ahead physical feasibility. The key distinction from the classical setting lies in the structure of $\mathcal{C}_t$: under a uniform matroid constraint, $\mathcal{C}_t$ coincides exactly with all feasible expansions $\mathcal{S}_t \to \mathcal{S}_{t+1}$ in $\mathcal{I}$, and the hereditary (M2) and symmetric exchange (M3) properties together guarantee that no optimal element is ever permanently excluded, requiring no probabilistic correction. Neither guarantee holds here. In a greedoid space under physical constraints, the reliance on a \emph{look-ahead physical feasibility oracle} restricts $\mathcal{C}_t$ to a strict subset of all feasible expansions, as illustrated in Fig.~\ref{fig:oracle_irregular}, and a sub-optimal greedy choice $e \in \mathcal{S}_t \setminus \mathcal{S}^{\star}$ may permanently exclude some optimal element $e^{\star} \in \mathcal{S}^{\star}$ from all future $\mathcal{C}_{t'}$, $t' > t$. We call this the \emph{unobserved items phenomenon}
(Definition~\ref{def:unobserved}).

\begin{defn}[Unobserved Items]
\label{def:unobserved}
An element $e \in \mathcal{P}$ is \emph{unobserved}, denoted $e \in \mathcal{U}$, if it is never placed in any candidate set $\mathcal{C}_t$ during Algorithm~\ref{alg:conditioned_greedy} execution. We write $m' := |\mathcal{P} \setminus \mathcal{U}|$ for the number of \emph{observed} elements. \boxend
\end{defn}

Since unobserved elements may include members of $\mathcal{S}^{\star}$, the optimality gap analysis carries a probabilistic dimension: the quality guarantee of Theorem~\ref{thm:gap} holds with a probability that degrades as more optimal elements are permanently excluded by the oracle, and this probabilistic correction is the sole additional cost imposed by the greedoid structure over the classical matroid case. Figure~\ref{fig:observability} illustrates this phenomenon on a six-node example running \texttt{FORWARD}: the observed set grows at each iteration but never covers $\mathcal{P}$ fully because the oracle restricts the candidate set at every step.

Let $\mathcal{S}^{\star} \in \mathcal{B}(\mathcal{I})$ denote an optimal solution to Problem~\eqref{eq:abstract} and $\mathcal{S}_t$ the partial solution after $t$ iterations of Conditioned Sequential Greedy Algorithm (Algorithm~\ref{alg:conditioned_greedy}). To establish the approximation guarantees of Algorithm~\ref{alg:conditioned_greedy}, we characterize the structural relationship between the current solution $\mathcal{S}_t$ and an optimal solution $\mathcal{S}^{\star}$. The following lemmas show that, whenever elements of $\mathcal{S}^{\star}$ appear in the candidate set, the exchange structure of $\mathcal{I}$ guarantees that they are compatible with $\mathcal{S}_t$ and can be used for greedy comparison.

\begin{lem}[Individual Augmentation]
\label{lem:individual}
At any iteration $t$, for every $e \in \mathcal{S}^{\star} \cap \mathcal{C}_t$, $\mathcal{S}_t \cup \{e\} \in \mathcal{I}$.
\end{lem}
\begin{proof}
By definition of the candidate set $\mathcal{C}_t$, every element $e \in \mathcal{C}_t$ satisfies $\mathcal{S}_t \cup \{e\} \in \mathcal{I}$. Since $e \in \mathcal{S}^{\star} \cap \mathcal{C}_t$, the claim follows immediately.
\end{proof}

\begin{lem}[Joint Augmentation]
\label{lem:joint}
At any iteration $t$,
\[
\mathcal{S}_t \cup
\bigl(\mathcal{S}^{\star} \cap \mathcal{C}_t\bigr)
\in \mathcal{I}.
\]
\end{lem}
\begin{proof}
Let $A := \mathcal{S}^{\star} \cap \mathcal{C}_t$. By Lemma~\ref{lem:individual}, for every $e \in A$ we have $\mathcal{S}_t \cup \{e\} \in \mathcal{I}$. Moreover, $A \subseteq \mathcal{S}^{\star}$ and $\mathcal{S}^{\star} \in \mathcal{I}$, hence all elements of $A$ coexist in the feasible set $\mathcal{S}^{\star}$.

Suppose, for contradiction, that $\mathcal{S}_t \cup A \notin \mathcal{I}$. Since $\mathcal{I}$ satisfies the greedoid accessibility property, there exists a nonempty subset $A' \subseteq A$ that is minimal with respect to inclusion such that $\mathcal{S}_t \cup A' \notin \mathcal{I}$. By minimality, for every $e \in A'$, we have $\mathcal{S}_t \cup (A' \setminus \{e\}) \in \mathcal{I}$.

Pick any $e \in A'$. Because $e \in A' \subseteq \mathcal{S}^{\star}$ and $\mathcal{S}^{\star} \in \mathcal{I}$, the greedoid exchange property guarantees that $e$ can be added to the feasible set $\mathcal{S}_t \cup (A' \setminus \{e\})$ without violating feasibility, a contradiction. Therefore, $\mathcal{S}_t \cup A \in \mathcal{I}$.
\end{proof}

Lemmas~\ref{lem:individual} and~\ref{lem:joint} guarantee that, conditioned on being observed, all elements of the optimal solution behave as admissible greedy alternatives. The remaining challenge is therefore not feasibility, but observability: we next show that, at each iteration, the candidate set contains at least one element of the optimal solution with quantifiable probability.

\begin{lem}[Non-emptiness of Optimal Augmentation Set]
\label{lem:nonempty}
Let $|\mathcal{P}| = m$ and let all bases of $(\mathcal{P}, \mathcal{I})$ have common cardinality $N-1$. At stage $t$,
\[
    \bigl\{e \in \mathcal{S}^{\star} \setminus \mathcal{S}_t
    \cap \mathcal{C}_t :
    \mathcal{S}_t \cup \{e\} \in \mathcal{I}\bigr\}
    \neq \emptyset
\]
with probability at least $P \geq 1 - |\mathcal{C}_t| \cdot \dfrac{N-1}{m - |\mathcal{C}_t \cup \mathcal{S}_t|}$.
\end{lem}
\begin{proof}
By Lemma~\ref{lem:individual}, feasibility of any element $e \in \mathcal{C}_t$ is guaranteed. Thus it suffices to show that $\mathcal{C}_t \cap (\mathcal{S}^{\star} \setminus \mathcal{S}_t) \neq \emptyset$.

We analyze this event under a uniform exposure model, in which the elements visible to the oracle at iteration $t$ are treated as being sampled uniformly at random from $\mathcal{P} \setminus \mathcal{S}_t$. By complementary probability and a union bound, we obtain
\[
\Pr\!\left[
\mathcal{C}_t \cap (\mathcal{S}^{\star} \setminus \mathcal{S}_t)
\neq \emptyset
\right]
\;\geq\;
1 - |\mathcal{C}_t|
\cdot
\frac{|\mathcal{S}^{\star} \setminus \mathcal{S}_t|}
{|\mathcal{P}| - |\mathcal{C}_t \cup \mathcal{S}_t|}.
\]

Since all bases have common cardinality $N-1$, we have $|\mathcal{S}^{\star} \setminus \mathcal{S}_t| \le N-1$, and the stated bound follows.
\end{proof}

\begin{rem}[Induction from an Arbitrary Subset]
\label{rem:induction}
Greedy convergence proofs typically start from the empty set, using the trivial inclusion $\emptyset \subseteq \mathcal{S}^{\star}$ to ensure a valid inductive base without prior knowledge. The same argument extends to any $\mathcal{S}_0 \subseteq \mathcal{S}^{\star}$ with $|\mathcal{S}_0| \leq |\mathcal{S}^{\star}|$; although this may seem restrictive, in practice the feasibility oracle defining $\mathcal{C}$ often produces elements aligned with $\mathcal{S}^{\star}$ early on. Thus, empty initialization is a worst-case baseline, while domain-informed initializations (e.g., in radial reconfiguration) effectively shift the inductive base, reducing the number of greedy steps and tightening the optimality gap without affecting correctness, see Fig.~\ref{fig:greedy_divergence_corrected}.
\end{rem}

\begin{figure}[t]
    \centering
    \begin{tikzpicture}[scale=0.75, every node/.style={font=\small}]
    \tikzstyle{optpt}=[circle, fill=black, inner sep=1.7pt]
    \tikzstyle{grpt}=[circle, fill=blue!70!black, inner sep=1.7pt]
    \tikzstyle{optline}=[black!75, thick]
    \tikzstyle{grline}=[blue!70!black, very thick]
    \tikzstyle{guide}=[gray!40, dashed]
    \coordinate (O0) at (0,0);
    \coordinate (O1) at (2,0);
    \coordinate (O2) at (4,0);
    \coordinate (O3) at (6,0);
    \coordinate (O4) at (8,0);
    \coordinate (O5) at (10,0);
    \draw[optline] (O0) -- (O5);
    \foreach \P in {O0,O1,O2,O3,O4,O5}
        \node[optpt] at (\P) {};
    \node[below=6pt] at (O0) {$\mathcal{S}_0^\star=\emptyset$};
    \node[below=6pt] at (O1) {$\mathcal{S}_1^\star$};
    \node[below=1pt] at (O2) {$\mathcal{S}_2^\star$};
    \node[below=6pt] at (O3) {$\mathcal{S}_3^\star$};
    \node[below=6pt] at (O4) {$\cdots$};
    \node[below=6pt] at (O5) {$\mathcal{S}_\kappa^\star=\mathcal{S}^\star$};
    \coordinate (G0) at (0,0);
    \coordinate (G1) at (2,0);
    \coordinate (G2) at (4,-0.7);
    \coordinate (G3) at (6,-0.2);
    \coordinate (G4) at (8,0.4);
    \coordinate (G5) at (10,0.8);
    \draw[grline]
    (G0) -- (G1)
    .. controls (3.0,0.0) and (4.2,-0.9) .. (G2)
    .. controls (5.0,-0.8) and (5.8,-0.3) .. (G3)
    .. controls (6.6,0.0) and (7.4,0.3) .. (G4)
    .. controls (8.6,0.6) and (9.2,0.8) .. (G5);
    \foreach \P in {G0,G1,G2,G3,G4,G5}
        \node[grpt] at (\P) {};
    \node[above=6pt, blue!70!black] at (G0) {$\mathcal{S}_0$};
    \node[above=6pt, blue!70!black] at (G1) {$\mathcal{S}_1$};
    \node[below=2pt, blue!70!black] at (G2) {$\mathcal{S}_2$};
    \node[below=11pt, blue!70!black] at (G3) {$\mathcal{S}_3$};
    \node[above=6pt, blue!70!black] at (G4) {$\cdots$};
    \node[right=2pt, blue!70!black] at (G5) {$\mathcal{S}_\kappa$};
    \draw[guide] (6.5,0.6) -- (6.5,-0.6);
    \node[above=12pt, blue!70!black] at (8.4,0.65)
        {\scriptsize accumulated error};
    \draw[->, thick] (-0.4,1.2) -- (10.5,1.2);
    \node[above] at (5,1.2) {\scriptsize greedy iterations};
    \end{tikzpicture}
\caption{Greedy sub-optimality intuition. Supermodularity
bounds the accumulated error at each iteration with respect
to an optimal solution.}
\label{fig:greedy_divergence_corrected}
\end{figure}
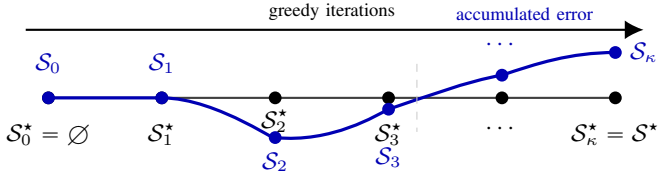

We are now ready to state the main optimality guarantee.

\begin{thm}[Conditioned Sequential Greedy Optimality over a Graphic Greedoid]
\label{thm:gap}
Let $(\mathcal{P}, \mathcal{I})$ be a graphic greedoid whose bases all have common cardinality $N-1$. Let $f : \mathcal{I} \to \mathbb{R}_{\geq 0}$ be normalized, monotone non-decreasing, and supermodular over $(\mathcal{P}, \mathcal{I})$ (Definitions~\ref{def:monotone}--\ref{def:supermodular}). Assume $\mathcal{S}_0 \subseteq \mathcal{S}^{\star}$ (Remark~\ref{rem:induction}) and let $a = \min_{t \in \{1,\ldots,\kappa\}} |\mathcal{C}_t|$ and $b = \max_{t \in \{1,\ldots,\kappa\}} |\mathcal{C}_t|$ be the minimum and maximum oracle space dimensions over all algorithm iterations. Then Algorithm~\ref{alg:conditioned_greedy} finds $\mathcal{S}_{N-1} \in \mathcal{B}(\mathcal{I})$ satisfying
\begin{equation}\label{eq:gap}
    f(\mathcal{S}_{N-1}) \;\leq\; \frac{b}{a}
    \cdot f(\mathcal{S}^{\star})
\end{equation}
with probability at least
\begin{equation}\label{eq:prob}
    P \;\geq\;
    1 - \bigl(m - (N-2)\bigr)
    \cdot \Bigl(1 - \tfrac{m'}{m}\Bigr),
\end{equation}
where $m = |\mathcal{P}|$ and $m' = |\mathcal{P} \setminus \mathcal{U}|$.
\end{thm}
\begin{proof}
We proceed by induction over the greedy iterations. Let $A_t := \mathcal{S}^{\star} \cap \mathcal{C}_t$ and set $k_t := |A_t|$.

\textbf{Step 1 (Greedy selection bound).}
At iteration $t$, the greedy rule selects
\[
s_t \in \arg\min_{e \in \mathcal{C}_t}
\Delta(e \mid \mathcal{S}_{t-1}),
\]
where
$\Delta(e \mid \mathcal{S}_{t-1})
= f(\mathcal{S}_{t-1} \cup \{e\}) - f(\mathcal{S}_{t-1})$.
By definition of the minimum,
\[
\Delta(s_t \mid \mathcal{S}_{t-1})
\;\le\;
\frac{1}{k_t}
\sum_{e \in A_t}
\Delta(e \mid \mathcal{S}_{t-1}).
\]

\textbf{Step 2 (Supermodularity).} Since $\mathcal{S}_{t-1} \subseteq \mathcal{S}_{t-1} \cup A_t$, supermodularity implies
\[
\Delta(e \mid \mathcal{S}_{t-1})
\;\le\;
\Delta(e \mid \mathcal{S}_{t-1} \cup A_t)
\quad \forall e \in A_t.
\]
Summing over $A_t$ and using Lemma~\ref{lem:joint},
\[
\Delta(s_t \mid \mathcal{S}_{t-1})
\;\le\;
\frac{1}{k_t}
\bigl(
f(\mathcal{S}_{t-1} \cup A_t)
- f(\mathcal{S}_{t-1})
\bigr).
\]

\textbf{Step 3 (Recursion).}
Using $f(\mathcal{S}_t) = f(\mathcal{S}_{t-1}) + \Delta(s_t \mid \mathcal{S}_{t-1})$, we obtain
\[
f(\mathcal{S}_t)
\;\le\;
\Bigl(1 - \frac{k_t}{|\mathcal{C}_t|}\Bigr)
f(\mathcal{S}_{t-1})
\;+\;
\frac{k_t}{|\mathcal{C}_t|}
f(\mathcal{S}_{t-1} \cup A_t).
\]

\textbf{Step 4 (Oracle size bound).}
By definition, $a \le |\mathcal{C}_t| \le b$. Conditioned on $A_t \neq \emptyset$ (Lemma~\ref{lem:nonempty}), we have $k_t \ge 1$ and hence
\[
\frac{1}{b}
\;\le\;
\frac{k_t}{|\mathcal{C}_t|}
\;\le\;
\frac{1}{a}.
\]
Using monotonicity and $\mathcal{S}_{t-1} \cup A_t \subseteq \mathcal{S}^{\star}$, the recursion gives
\[
f(\mathcal{S}_t)
\;\le\;
\Bigl(1 - \frac{1}{b}\Bigr)
f(\mathcal{S}_{t-1})
+ \frac{1}{a} f(\mathcal{S}^{\star}).
\]

Unrolling this inequality for $t = 1,\dots,N-1$ and using $\mathcal{S}_0 \subseteq \mathcal{S}^{\star}$, we obtain
\[
f(\mathcal{S}_{N-1})
\;\le\;
\frac{b}{a}\,f(\mathcal{S}^{\star}).
\]

\textbf{Step 5 (Probabilistic correction).}
By Lemma~\ref{lem:nonempty}, at each iteration the event $A_t \neq \emptyset$ fails with probability at most $1 - m'/m$. Applying a union bound over the $N-2$ greedy iterations yields
\[
P
\;\ge\;
1 - \bigl(m - (N-2)\bigr)
\Bigl(1 - \frac{m'}{m}\Bigr),
\] which establishes~\eqref{eq:prob}.
\end{proof}

Theorem~\ref{thm:gap} reveals two independent axes along which Algorithm~\ref{alg:conditioned_greedy} may deviate from the global optimum. The ratio $b/a$ captures the variability of the oracle space across iterations: when the candidate set is uniform in size ($a = b$), the bound reduces to $f(\mathcal{S}_{N-1}) \leq f(\mathcal{S}^{\star})$, recovering exact optimality. The probabilistic factor~\eqref{eq:prob} captures the structural price of feasibility conditioning: when $m' = m$, meaning all elements are observed at some iteration, the guarantee becomes fully deterministic. In practice, denser and more connected networks expose a larger fraction of $\mathcal{P}$ to the oracle across iterations, pushing both $m' \to m$ and $a \to b$ simultaneously, so the two bounds tighten together as the network becomes richer.

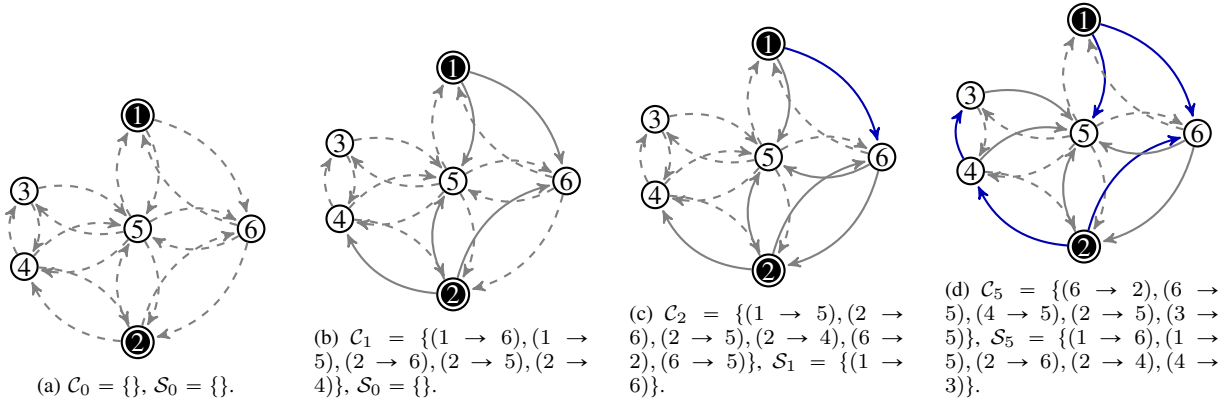
\begin{figure*}[t]
\centering
    \begin{subfigure}[b]{0.2\textwidth}
        \centering
        \begin{tikzpicture}[>=stealth',shorten >=1pt,auto,node distance=1.5cm,
      thick,main node/.style={circle,fill=black,draw,minimum size=5pt,inner sep=0pt}]
          \node[main node, text=white, minimum size=8pt] (x1) at (0,1.5) {1};
          \node[main node, fill=none, minimum size=12pt] (x1111) at (0,1.5) {};
          \node[main node, text=white, minimum size=8pt] (x2) at (0,-1.5) {2};
          \node[main node, fill=none, minimum size=12pt] (x2222) at (0,-1.5) {};
          \node[main node, fill=white, minimum size=10pt] (x3) at (-1.5,0.5) {3};
          \node[main node, fill=white, minimum size=10pt] (x4) at (-1.5,-0.5) {4};
          \node[main node, fill=white, minimum size=10pt] (x5) at (0,0) {5};
          \node[main node, fill=white, minimum size=10pt] (x6) at (1.5,0) {6};
          \path[every node/.style={font=\sffamily\small}]
            (x1111) edge[bend left,->,gray,dashed] (x5)
            (x5) edge[bend left,->,gray,dashed] (x1111)
            (x2222) edge[bend left,->,gray,dashed] (x5)
            (x5) edge[bend left,->,gray,dashed] (x2222)
            (x1111) edge[bend left,->,gray,dashed] (x6)
            (x6) edge[bend left,->,gray,dashed] (x1111)
            (x2222) edge[bend left,->,gray,dashed] (x6)
            (x6) edge[bend left,->,gray,dashed] (x2222)
            (x5) edge[bend left,->,gray,dashed] (x6)
            (x6) edge[bend left,->,gray,dashed] (x5)
            (x2222) edge[bend left,->,gray,dashed] (x4)
            (x4) edge[bend left,->,gray,dashed] (x2222)
            (x4) edge[bend left,->,gray,dashed] (x5)
            (x5) edge[bend left,->,gray,dashed] (x4)
            (x4) edge[bend left,->,gray,dashed] (x3)
            (x3) edge[bend left,->,gray,dashed] (x4)
            (x5) edge[bend left,->,gray,dashed] (x3)
            (x3) edge[bend left,->,gray,dashed] (x5);
        \end{tikzpicture}
        \caption{$\mathcal{C}_0=\{\}$,
                 $\mathcal{S}_0=\{\}$.}
        \label{fig:observe1}
    \end{subfigure}
    \hspace{3mm}
    \begin{subfigure}[b]{0.2\textwidth}
        \centering
        \begin{tikzpicture}[>=stealth',shorten >=1pt,auto,node distance=1.5cm,
      thick,main node/.style={circle,fill=black,draw,minimum size=5pt,inner sep=0pt}]
          \node[main node, text=white, minimum size=8pt] (x1) at (0,1.5) {1};
          \node[main node, fill=none, minimum size=12pt] (x1111) at (0,1.5) {};
          \node[main node, text=white, minimum size=8pt] (x2) at (0,-1.5) {2};
          \node[main node, fill=none, minimum size=12pt] (x2222) at (0,-1.5) {};
          \node[main node, fill=white, minimum size=10pt] (x3) at (-1.5,0.5) {3};
          \node[main node, fill=white, minimum size=10pt] (x4) at (-1.5,-0.5) {4};
          \node[main node, fill=white, minimum size=10pt] (x5) at (0,0) {5};
          \node[main node, fill=white, minimum size=10pt] (x6) at (1.5,0) {6};
          \path[every node/.style={font=\sffamily\small}]
            (x1111) edge[bend left,->,gray] (x5)
            (x5) edge[bend left,->,gray,dashed] (x1111)
            (x2222) edge[bend left,->,gray] (x5)
            (x5) edge[bend left,->,gray,dashed] (x2222)
            (x1111) edge[bend left,->,gray] (x6)
            (x6) edge[bend left,->,gray,dashed] (x1111)
            (x2222) edge[bend left,->,gray] (x6)
            (x6) edge[bend left,->,gray,dashed] (x2222)
            (x5) edge[bend left,->,gray,dashed] (x6)
            (x6) edge[bend left,->,gray,dashed] (x5)
            (x2222) edge[bend left,->,gray] (x4)
            (x4) edge[bend left,->,gray,dashed] (x2222)
            (x4) edge[bend left,->,gray,dashed] (x5)
            (x5) edge[bend left,->,gray,dashed] (x4)
            (x4) edge[bend left,->,gray,dashed] (x3)
            (x3) edge[bend left,->,gray,dashed] (x4)
            (x5) edge[bend left,->,gray,dashed] (x3)
            (x3) edge[bend left,->,gray,dashed] (x5);
        \end{tikzpicture}
        \caption{$\mathcal{C}_1=\{(1\to6),(1\to5),(2\to6),
                 (2\to5),(2\to4)\}$,
                 $\mathcal{S}_0=\{\}$.}
        \label{fig:observe2}
    \end{subfigure}
    \hspace{3mm}
    \begin{subfigure}[b]{0.2\textwidth}
        \centering
        \begin{tikzpicture}[>=stealth',shorten >=1pt,auto,node distance=1.5cm,
      thick,main node/.style={circle,fill=black,draw,minimum size=5pt,inner sep=0pt}]
          \node[main node, text=white, minimum size=8pt] (x1) at (0,1.5) {1};
          \node[main node, fill=none, minimum size=12pt] (x1111) at (0,1.5) {};
          \node[main node, text=white, minimum size=8pt] (x2) at (0,-1.5) {2};
          \node[main node, fill=none, minimum size=12pt] (x2222) at (0,-1.5) {};
          \node[main node, fill=white, minimum size=10pt] (x3) at (-1.5,0.5) {3};
          \node[main node, fill=white, minimum size=10pt] (x4) at (-1.5,-0.5) {4};
          \node[main node, fill=white, minimum size=10pt] (x5) at (0,0) {5};
          \node[main node, fill=white, minimum size=10pt] (x6) at (1.5,0) {6};
          \path[every node/.style={font=\sffamily\small}]
            (x1111) edge[bend left,->,gray] (x5)
            (x5) edge[bend left,->,gray,dashed] (x1111)
            (x2222) edge[bend left,->,gray] (x5)
            (x5) edge[bend left,->,gray,dashed] (x2222)
            (x1111) edge[bend left,->,blue!70!black] (x6)
            (x6) edge[bend left,->,gray,dashed] (x1111)
            (x2222) edge[bend left,->,gray] (x6)
            (x6) edge[bend left,->,gray] (x2222)
            (x5) edge[bend left,->,gray,dashed] (x6)
            (x6) edge[bend left,->,gray] (x5)
            (x2222) edge[bend left,->,gray] (x4)
            (x4) edge[bend left,->,gray,dashed] (x2222)
            (x4) edge[bend left,->,gray,dashed] (x5)
            (x5) edge[bend left,->,gray,dashed] (x4)
            (x4) edge[bend left,->,gray,dashed] (x3)
            (x3) edge[bend left,->,gray,dashed] (x4)
            (x5) edge[bend left,->,gray,dashed] (x3)
            (x3) edge[bend left,->,gray,dashed] (x5);
        \end{tikzpicture}
        \caption{$\mathcal{C}_2=\{(1\to5),(2\to6),(2\to5),
                 (2\to4),(6\to2),(6\to5)\}$,
                 $\mathcal{S}_1=\{(1\to6)\}$.}
        \label{fig:observe3}
    \end{subfigure}
    \hspace{3mm}
    \begin{subfigure}[b]{0.2\textwidth}
        \centering
        \begin{tikzpicture}[>=stealth',shorten >=1pt,auto,node distance=1.5cm,
      thick,main node/.style={circle,fill=black,draw,minimum size=5pt,inner sep=0pt}]
          \node[main node, text=white, minimum size=8pt] (x1) at (0,1.5) {1};
          \node[main node, fill=none, minimum size=12pt] (x1111) at (0,1.5) {};
          \node[main node, text=white, minimum size=8pt] (x2) at (0,-1.5) {2};
          \node[main node, fill=none, minimum size=12pt] (x2222) at (0,-1.5) {};
          \node[main node, fill=white, minimum size=10pt] (x3) at (-1.5,0.5) {3};
          \node[main node, fill=white, minimum size=10pt] (x4) at (-1.5,-0.5) {4};
          \node[main node, fill=white, minimum size=10pt] (x5) at (0,0) {5};
          \node[main node, fill=white, minimum size=10pt] (x6) at (1.5,0) {6};
          \path[every node/.style={font=\sffamily\small}]
            (x1111) edge[bend left,->,blue!70!black] (x5)
            (x5) edge[bend left,->,gray,dashed] (x1111)
            (x2222) edge[bend left,->,gray] (x5)
            (x5) edge[bend left,->,gray,dashed] (x2222)
            (x1111) edge[bend left,->,blue!70!black] (x6)
            (x6) edge[bend left,->,gray,dashed] (x1111)
            (x2222) edge[bend left,->,blue!70!black] (x6)
            (x6) edge[bend left,->,gray] (x2222)
            (x5) edge[bend left,->,gray,dashed] (x6)
            (x6) edge[bend left,->,gray] (x5)
            (x2222) edge[bend left,->,blue!70!black] (x4)
            (x4) edge[bend left,->,gray,dashed] (x2222)
            (x4) edge[bend left,->,gray] (x5)
            (x5) edge[bend left,->,gray,dashed] (x4)
            (x4) edge[bend left,->,blue!70!black] (x3)
            (x3) edge[bend left,->,gray,dashed] (x4)
            (x5) edge[bend left,->,gray,dashed] (x3)
            (x3) edge[bend left,->,gray] (x5);
        \end{tikzpicture}
        \caption{$\mathcal{C}_5=\{(6\to2),(6\to5),(4\to5),
                 (2\to5),(3\to5)\}$,
                 $\mathcal{S}_5=\{(1\to6),(1\to5),(2\to6),
                 (2\to4),(4\to3)\}$.}
        \label{fig:observe4}
    \end{subfigure}
\caption{Observability in \texttt{FORWARD}.
\textcolor{gray}{$\dashrightarrow$} unobserved edges;
\textcolor{gray}{$\rightarrow$} observed edges;
\textcolor{blue!70!black}{$\rightarrow$} selected edges.
The observed set grows at each iteration but never covers
$\mathcal{P}$ fully due to the oracle's feasibility
restrictions.}
\label{fig:observability}
\end{figure*}


\section{Case Study: \texttt{FORWARD} as a Canonical Instance}
\label{sec:forward}

We now instantiate the abstract framework for the multi-source radial reconfiguration problem and show that \texttt{FORWARD} Algorithm, introduced in~\cite{JV-RB-SK:25:ACC} and later extended in~\cite{JV-RB-SK:25arxiv}, is a concrete realization of Algorithm~\ref{alg:conditioned_greedy}.

\subsection{Problem Formulation}
Consider a distribution network $\mathcal{G}_D = (\mathcal{V}_D, \mathcal{E}_D)$ with $N = |\mathcal{V}_D|$ nodes, $m = |\mathcal{E}_D|$ edges, $n_g$ source nodes $\mathcal{V}_g$, and $n_c$ sink nodes $\mathcal{V}_c$ with $\mathcal{V}_g\cup\mathcal{V}_c= \mathcal{V}_D$ and $\mathcal{V}_g\cap\mathcal{V}_c=\emptyset$. The Optimal Radial Reconfiguration Problem is:
\begin{subequations}\label{eq:network}
\begin{align}
    \min_{\mathcal{S},\,\mathbf{x}} \quad
    & f(\mathcal{S}, \mathbf{x})
    = \sum_{(i,j) \in \mathcal{S}}
    C_{i,j} \cdot x_{i,j}^2 \\
    \text{s.t.} \quad
    & \mathcal{G}(\mathcal{V}_D, \mathcal{S})
    \in \mathcal{Q}, \\
    & A(\mathcal{S})\,\mathbf{x}
    = \mathbf{g} - \mathbf{d}, \\
    & \mathbf{0} \leq \mathbf{x} \leq \bar{\mathbf{x}},
    \label{eqn:capacity_constraint}
\end{align}
\end{subequations}
where $\mathcal{Q}$ is the set of all spanning radial polyforests rooted at $\mathcal{V}_g$, which is a problem captured by Problem~\eqref{eq:abstract}, see Remark~\ref{rem:instantiation_problem4}.

\begin{rem}
  \label{rem:instantiation_problem4}
  Problem~\eqref{eq:network} is an instance of Problem~\eqref{eq:abstract} under the identification:
    \begin{itemize}
        \item $\mathcal{P} = \mathcal{E}_D$;
        \item $\mathcal{B}(\mathcal{I}) = \mathcal{Q}$, forming an $\mathcal{R}$-rooted spanning forest greedoid in the sense of Definition~\ref{def:greedoid} with root set $\mathcal{R} = \mathcal{V}_g$: the feasible collection $\mathcal{I}$ consists of all acyclic subsets of $\mathcal{E}_D$ in which every node has in-degree at most one and every connected component contains at least one node from $\mathcal{V}_g$, and the bases $\mathcal{B}(\mathcal{I}) = \mathcal{Q}$ are precisely the spanning radial polyforests rooted at $\mathcal{V}_g$;
        \item $g\!\left(\mathbf{x}(\mathcal{S})\right):= A(\mathcal{S})\mathbf{x}(\mathcal{S}) - (\mathbf{g}-\mathbf{d}) = \mathbf{0}$, uniquely resolved by Kirchhoff's current law on the tree~\cite[Lemma 4]{JV-RB-SK:25arxiv};
       \item $f(\mathcal{S})$ is monotone non-decreasing and supermodular over $(\mathcal{E}_D, \mathcal{I})$\footnote{Detailed validation is given in the appendix.}.
    \end{itemize}
\end{rem}

\subsection{The \texttt{FORWARD} Algorithm}

\texttt{FORWARD}~\cite{JV-RB-SK:25arxiv} is a polynomial-time algorithm that builds a spanning radial polyforest edge by edge while certifying structural and physical feasibility at every step. It operates on the $\mathcal{R}$-rooted spanning forest greedoid $(\mathcal{E}_D, \mathcal{I})$ of Definition~\ref{def:greedoid} with root set $\mathcal{R} = \mathcal{V}_g$ and is driven by five components:
\begin{itemize}
    \item \textsf{Pre-Processor}: Iteratively processes pendant nodes, directing mandatory edges and redistributing their nodal values to parent nodes, returning the 2-core subgraph $\mathcal{G}_P$ and updated source set $\mathcal{V}_g$ as outputs. This eliminates trivial components that must appear in any feasible solution before greedy expansion begins~\cite[Algorithm~3]{JV-RB-SK:25arxiv}.
    \item \textsf{Islander}: Partitions $\mathcal{G}_P$ at articulation \emph{source} nodes into disjoint balanced subgraphs $\mathcal{G}^\ell$, enabling parallel processing while preserving both feasibility and optimality~\cite[Theorem~5, Corollary~6]{JV-RB-SK:25arxiv}.
    \item \textsf{Net-Concad} (Network Condensation): At each iteration $t$, condenses previously sampled polytrees into super source nodes and unsampled sink components into super sink nodes, forming a quasi-bipartite dual graph $\bar{\mathcal{G}}^\ell$. The edges of $\bar{\mathcal{G}}^\ell$ define the current candidate set $\mathcal{C}_t$, enforcing greedoid membership $\bar{\mathcal{S}} \cup \{e\} \in \mathcal{I}$ and certifying that each candidate admits a completion to a spanning polyforest rooted at $\mathcal{V}_g$, thereby implementing $\mathcal{F}(\bar{\mathcal{S}} \cup \{e\}) = 1$. This mechanism addresses the shortsightedness of greedy selection by providing a global view of remaining demand.
    \item \textsf{Sampler}: Selects from $\mathcal{C}_t$ using a priority queue sorted by physically-motivated weights $w_{i,j}$ defined by edge resistance and rolling demand estimates~\cite[Algorithm~6, Eq.~(2)]{JV-RB-SK:25arxiv}, prioritizing pendant super unsampled nodes and edges with sufficient capacity. This implements the greedy selection of Algorithm~\ref{alg:conditioned_greedy} through a physics-aware surrogate for marginal cost.
    \item \textsf{Rewire}: When active capacity constraints~\eqref{eqn:capacity_constraint} produce unsupplied nodes, performs strategic edge swaps to redistribute flow from surplus source-to-leaf paths to deficit regions, restoring feasibility while maintaining radiality~\cite[Lemma~10]{JV-RB-SK:25arxiv}.
\end{itemize}
Together, \textsf{Pre-Processor}, \textsf{Islander}, and \textsf{Net-Concad} constitute a computationally efficient implementation of the look-ahead physical feasibility oracle $\mathcal{F}$ of Definition~\ref{def:oracle}, while \textsf{Sampler} implements the greedy selection step of Algorithm~\ref{alg:conditioned_greedy}.

A key theoretical property of \texttt{FORWARD} is that it comes with a \emph{certified feasible final solution}. Specifically, by~\cite[Theorem~9]{JV-RB-SK:25arxiv}, \texttt{FORWARD} without capacity constraints is guaranteed to terminate in a feasible spanning radial polyforest $\mathcal{S} \in \mathcal{Q}$ satisfying $A(\mathcal{S})\mathbf{x}(\mathcal{S}) = \mathbf{g} - \mathbf{d}$. When capacity constraints are active, the \textsf{Rewire} component resolves any residual infeasibilities through strategic flow redistribution, and by~\cite[Theorem~11]{JV-RB-SK:25arxiv}, the complete \texttt{FORWARD} algorithm is guaranteed to return a feasible solution under Assumption~1 of~\cite{JV-RB-SK:25arxiv} (non-empty feasible set). This distinguishes \texttt{FORWARD} from heuristics that may terminate in infeasible configurations, and is precisely what makes it a well-posed canonical instance of Algorithm~\ref{alg:conditioned_greedy}.

\subsection{Mapping to the Abstract Framework}
Now let us demonstrate \texttt{FORWARD} is essentially a Conditioned Sequential Greedy Algorithm when \eqref{eqn:capacity_constraint} is not activated.

\begin{rem}[\texttt{FORWARD} as a Conditioned Sequential Greedy Instance]
\label{rem:instance}
    The feasibility oracle $\mathcal{F}$ is implemented jointly by \textsf{Pre-Processor}, \textsf{Islander}, and \textsf{Net-Concad}; the greedy selection is implemented by \textsf{Sampler}; and if capacity constraint~\eqref{eqn:capacity_constraint} is inactive, \textsf{Rewire} is not applied, maintaining a sequential greedy implementation. \textsf{Rewire}, when active, is equivalent to a Local Greedy Search~\cite{AK-GY-SB-EV-MC-TL-EP:18}, see Lemma~\ref{lem:rewire_equivalence}.
\end{rem}

\begin{rem}[Oracle Simplification in \texttt{FORWARD}]
\label{rem:oracle_simplification}
For every spanning radial polyforest $\mathcal{S} \in \mathcal{Q}$, flow conservation $g(\mathbf{x}(\mathcal{S})) = 0$ holds automatically by Kirchhoff's law~\cite{JV-RB-SK:25arxiv}. Therefore, $\mathcal{F}(\mathcal{S} \cup \{e\}) = 1$ reduces to a pure greedoid completability check, making \textsf{Pre-Processor}, \textsf{Islander}, and \textsf{Net-Concad} a computationally efficient implementation of Definition~\ref{def:oracle}.
\end{rem}

\begin{lem}\longthmtitle{\textsf{Rewire} Equivalence to Local Greedy Search}
    \label{lem:rewire_equivalence}
    The \textsf{Rewire} sub-routine is equivalent to performing a Local Greedy Search such as \cite[Algorithm 1]{AK-GY-SB-EV-MC-TL-EP:18}.
\end{lem}
\begin{proof}
    Given a configuration $\mathcal{S}$, \textsf{Rewire} swaps sampled edges $e\in\mathcal{S}$ with non-sampled edges $\bar{e}\in\mathcal{E}_D\setminus\mathcal{S}$ in order to achieve an improved solution $\mathcal{S}'$. By the function definition in Section~\ref{sec:preliminaries}, since the initial configuration is infeasible, $f(\mathcal{S})=\infty$. Hence, the \textsf{Rewire} operation is equivalent to modifying the current solution $\mathcal{S}$ if $f((\mathcal{S}\setminus\{e\})\cup\{\bar{e}\})<(1-\epsilon)f(\mathcal{S})$, as in the Local Greedy Search of~\cite{AK-GY-SB-EV-MC-TL-EP:18}, concluding the proof.
\end{proof}

After establishing \texttt{FORWARD} as a canonical instance of Algorithm~\ref{alg:conditioned_greedy}, we apply Theorem~\ref{thm:gap} to Problem~\eqref{eq:network}.

\begin{cor}[Optimality Gap of \texttt{FORWARD}]
\label{cor:forward}
When capacity constraints are inactive, \texttt{FORWARD} finds $\mathcal{S}_{N-1} \in \mathcal{Q}$ satisfying
$$
    f(\mathcal{S}_{N-1}) \leq \frac{b}{a} \cdot
    f(\mathcal{S}^{\star})
$$
with probability at least $P \geq 1 - (m-(N-2))\cdot(1-m'/m)$. When capacity constraints are active, \textsf{Rewire} is equivalent to a Local Greedy Search and its guarantee follows from the corresponding local search approximation result~\cite{AK-GY-SB-EV-MC-TL-EP:18}.
\end{cor}

\begin{proof}
Follows directly from Theorem~\ref{thm:gap}, Remark~\ref{rem:instance}, and Lemma~\ref{lem:rewire_equivalence}.
\end{proof}

\subsection{Numerical Study}
To demonstrate the tightness of the proposed optimality gap in Theorem~\ref{thm:gap}, we have compared it against the \texttt{FORWARD} analytical solution among $4$ distinct distribution networks: IEEE $13$, IEEE $18$, IEEE $33$ and IEEE $37$, where the number indicate the amount of nodes in each network. The case study, as explained in Section~\ref{sec:intro}, has been reduced to small networks in order to compare \texttt{FORWARD} performance with MINLP methods exemplified on \texttt{KNITRO}. As shown in \cite{JV-RB-SK:25}, these kind of commercial methods commonly fail in obtaining a solution for greater size networks.

\begin{table}[t]
\centering
\caption{{Comparison between \texttt{FORWARD} analytical performance and worst-case optimality gap, Theorem~\ref{thm:gap}.}}
\label{tab:sample_table}
{\scriptsize
\begin{tabular}{|c|c|c|c|c|c|c|}
\hline
Network & $m$ & $m^\prime$ & $f(\mathcal{S})$ & $\leq f(\mathcal{S}^\star)$ & $f(\hat{\mathcal{S}}^\star)$ & $P_\geq$ \\ 
\hline
IEEE 13 & 13 & 13 & 360.183 & 196.463 & 360.183 & 1 \\ 
\hline
IEEE 18 & 19 & 19 & 175.821 & 93.0817 & 175.821 & 1 \\ 
\hline
IEEE 33 & 38 & 36 & 594.282 & 305.172 & 318.568 & 0.6316\\ 
\hline
IEEE 37 & 48 & 45 & 598.720 & 399.147 & 623.54 & 0.1875\\
\hline
\end{tabular}
}
\end{table}

As shown in Table~\ref{tab:sample_table}, the proposed gap serves as an accurate optimality gap for \texttt{FORWARD} method. It is known that MINLP methods does not guarantee optimality in the solution, however, its recursive implementation has been used to find the minimum loss configuration experimentally which we use as an optimal estimation $f(\hat{\mathcal{S}}^\star)$. On the other hand, Theorem~\ref{thm:gap} can be rewritten such that
\begin{equation*}
        \frac{a}{b}\cdot f(\mathcal{S}_{N-1}) \leq f(\mathcal{S}^\star)
\end{equation*}
obtaining the expected best case optimal value. By stacking to the obtained results, we can see that with high-probability \texttt{FORWARD} achieves a $2$-approximation solution with respect to the optimal in the worst-case scenario and practically a $1.5$-approximation in the best-case (considering MINLP as the optimal solution).

\begin{figure}[t]
\centering
    \begin{tikzpicture}
    \begin{axis}[
        scale=0.8,
        width=0.95\linewidth,
        height=6.2cm,
        xlabel={$m^\prime/m$},
        ylabel={$P$},
        xmin=0.8, xmax=1.0,
        ymin=-7.8, ymax=1.5,
        domain=0.8:1.0,
        samples=200,
        grid=both,
        minor tick num=1,
        tick label style={font=\small},
        label style={font=\small},
        legend style={
            at={(0.98,0.02)},
            anchor=south east,
            font=\small,
            draw=black,
            fill=white
        },
    ]
    
    \addplot[dashed, thick, black] {0};
    
    
    \addplot[thick, blue] {1 - (2)*(1-x)};
    \addlegendentry{$c=1$}
    
    \addplot[thick, orange] {1 - (12)*(1-x)};
    \addlegendentry{$c=2$}
    
    \addplot[thick, green!60!black] {1 - (22)*(1-x)};
    \addlegendentry{$c=3$}
    
    \addplot[thick, red] {1 - (32)*(1-x)};
    \addlegendentry{$c=4$}
    
    \addplot[thick, purple] {1 - (42)*(1-x)};
    \addlegendentry{$c=5$}
    
    \end{axis}
    \end{tikzpicture}
    \caption{Probability bound $P\geq 1- ((c-1)N+2)\cdot(1-\frac{m^\prime}{m})$ as a function of the ratio $m'/m$ for different values of $c$.}
    \label{fig:probability_bound}
\end{figure}
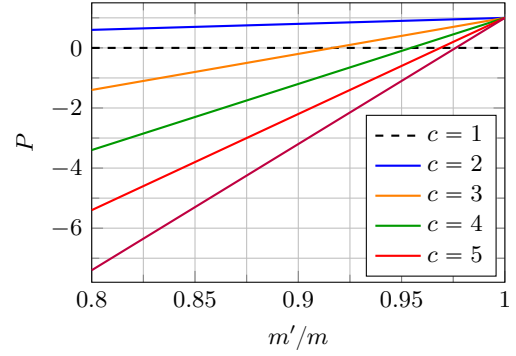

It is important to observe that Theorem~\ref{thm:gap} is simplified by considering a probabilistic approximation. Hence, it has also been studied how does the probability guarantees evolve with network size. Notice that the object of study are small-world networks where $m=c\cdot N$. As shown in Figure~\ref{fig:probability_bound}, the reliability increases linearly with respect to the percentage of observed edges with respect to the complete edge space size. Additionally, the more density in the network, the steepest is the evolution, which is expected as the more edges, the more edges will remain unobserved. It is important to highlight that the approximation guarantees will only hold with high-probability for a certain range of graph properties.

\section{Conclusions}
\label{sec:conclusion}

This work formalizes the interplay between combinatorial structure and physical feasibility in greedy optimization by introducing a conditioned sequential greedy framework over greedoid equipped with a look-ahead feasibility oracle. We show that enforcing extendability through oracle-restricted candidate sets fundamentally alters the search process, giving rise to a quantifiable \emph{price of feasibility}. The derived approximation bound captures this effect through the variability of oracle space dimensions, complemented by a probabilistic correction accounting for unobserved elements in the search space. By establishing that the FORWARD algorithm is a canonical instance of the proposed framework, we provide its first theoretical optimality guarantees, bridging the gap between empirical performance and formal analysis. Overall, the results highlight a fundamental trade-off between feasibility preservation and optimality, offering a principled lens to design and analyze greedy algorithms in cyber-physical systems where constraints are global, non-local, and only verifiable at completion. Future work will explore adaptive and learned oracle constructions to reduce conservatism and tighten the feasibility-optimality trade-off.


\bibliographystyle{ieeetr}
\bibliography{main}

\appendices
\section{Cost Evaluation for Radial Configurations via the
Flow-Processor Oracle}
\label{app:flow_processor}

The objective in~\eqref{eq:network} is evaluated over the $\mathcal{R}$-rooted spanning forest greedoid $(\mathcal{E}_D,\mathcal{I})$ of Definition~\ref{def:polyforest_greedoid}, equipped with the look-ahead physical feasibility oracle of Definition~\ref{def:oracle}. A central mechanism in this evaluation is the \emph{nodal value}, which combines supply and demand in a single state variable and governs the incremental accumulation of flows as a spanning radial polyforest is constructed edge by edge. Accordingly, the evaluation of $f(\mathcal{S})$ accounts for the construction history of $\mathcal{S}$. We first introduce nodal values and their incremental flow semantics.

\begin{defn}[Nodal Value Vector]
\label{def:nodal_value}
The nodal value vector $\mathbf{p}\in\mathbb{R}^N$ is defined by
\[
p_i =
\begin{cases}
g_i>0, & i\in\mathcal{V}_g \quad \text{(source)},\\
-d_i<0, & i\in\mathcal{V}_c \quad \text{(sink)},
\end{cases}
\]
with global balance $\sum_{i\in\mathcal{V}_D}p_i=0$. At any stage of the construction, $p_i>0$ denotes residual supply at node $i$, whereas $p_i<0$ denotes unmet demand. A sink is fully supplied if and only if $p_i=0$. \boxend
\end{defn}

When a directed edge $e^\star=(i\to j)$ is sampled, the flow assignment is governed by the residual imbalance at its endpoints. Define
\[
\delta:=\min\{p_i,|p_j|\}.
\]
The edge flow is set to $x_{i,j}\leftarrow\delta$, and the nodal values are updated according to
\[
p_i\leftarrow p_i-\delta,
\qquad
p_j\leftarrow p_j+\delta.
\]
If $\delta=|p_j|$, node $j$ becomes fully supplied; otherwise, $p_j<0$ records its remaining unmet demand. These updates enforce flow conservation while preserving the interpretation of nodal values as residual imbalances.

\subsection{Flow Propagation in Multi-Source Polytrees}

Once $\delta$ units of flow are injected at node $j$, they must be routed through the existing polyforest. Because each polytree is directed and acyclic, there is a unique directed path from a supplying source to $i$ and from $j$ to any downstream node. Flow therefore propagates along these unique paths.

Nodal values are essential for \emph{supplier identification}. In a multi-source polytree, topology alone does not identify the source of a given unit of flow. The nodal values remove this ambiguity: the flow $\delta$ is attributed to the source, or relay source, whose residual nodal value decreases by $\delta$. Every edge on the unique directed path from that source to $i$ is then incremented by $\delta$.

\subsection{Merging Two Existing Polytrees}

Suppose a newly sampled edge connects two distinct polytrees $\mathcal{T}_i$ and $\mathcal{T}_j$. Two cases arise, depending on where the injected supply is absorbed.

\paragraph{Case 1: Local shortfall at the connecting node.}
If $p_j<0$, the new edge supplies the shortfall at $j$ directly. The amount $\delta=\min\{p_i,|p_j|\}$ is injected at $j$, and the nodal values are updated as above. Every edge on the unique path from the supplying source in $\mathcal{T}_i$ to node $i$ is incremented by $\delta$. This is the same update used when attaching an unsampled sink component.

\paragraph{Case 2: Downstream shortfall in the receiving polytree.}
If $p_j\geq0$ but $\mathcal{T}_j$ contains unmet demand, then its aggregate nodal value
\[
\bar p_{\mathcal{T}_j}
:=\sum_{v\in\mathcal{T}_j}p_v
\]
is negative. The new edge does not serve $j$ directly; instead, it allows supply to reach a deficit node $k\in\mathcal{T}_j$ through $j$. The injected amount
\[
\delta
=\min\{\bar p_{\mathcal{T}_i},
          |\bar p_{\mathcal{T}_j}|\}
\]
is propagated along two unique directed paths: upstream from the supplying source in $\mathcal{T}_i$ to $i$, and downstream within $\mathcal{T}_j$ from $j$ to $k$. Every edge on these paths is incremented by $\delta$; all other edge flows remain unchanged.

In both cases, the update preserves acyclicity, respects the greedoid constraints of Definition~\ref{def:polyforest_greedoid}, and maintains flow conservation throughout the polyforest.

\subsection{Downstream-Demand Characterization}
At any intermediate stage, the flow on an edge equals the total demand assigned downstream through that edge:
\begin{equation}
x_e(\mathcal{S}) =\sum_{k\in\mathcal{V}_c(\mathcal{S},e)}d_k,
\label{eq:edge_flow}
\end{equation}
where $\mathcal{V}_c(\mathcal{S},e)$ is the set of sink nodes whose unique source path traverses $e$. These downstream sink sets can only grow as edges are added; they never shrink.

Upon termination at a complete spanning radial polyforest $\mathcal{S}\in\mathcal{Q}$, the accumulated flows coincide with the unique solution $\mathbf{x}(\mathcal{S})$ of the balance equation~\eqref{eq:network}, as guaranteed by~\cite[Lemma~4]{JV-RB-SK:25arxiv}.

\subsection{Oracle-Induced Flow Evaluation}

\begin{defn}[Oracle-Induced Flow Realization]
\label{def:oracle_flow}
For every $\mathcal{S}\in\mathcal{Q}$, the flow vector $\mathbf{x}(\mathcal{S})$ is the unique output of the pendant-node inward traversal described above. All marginal costs
\[
\Delta(\mathcal{S},e)
:=f(\mathcal{S}\cup\{e\})-f(\mathcal{S})
\]
are evaluated with respect to this oracle-induced realization. \boxend
\end{defn}

This traversal computes $\mathbf{x}(\mathcal{S})$ and $f(\mathcal{S})$ for every complete $\mathcal{S}\in\mathcal{Q}$ and is equivalent to the Tree-Processor of~\cite{JV-RB-SK:25arxiv}, augmented with cost evaluation.

\subsection{Implications for Supermodularity}

The monotone expansion of the downstream sink sets in \eqref{eq:edge_flow}, together with the quadratic cost, implies that marginal flow increments grow as the base polyforest expands. Hence, under the oracle-induced realization of Definition~\ref{def:oracle_flow}, $f$ is monotone nondecreasing and supermodular in the senses of Definitions~\ref{def:monotone} and~\ref{def:supermodular}, respectively.

\begin{rem}[Oracle-Dependent Supermodularity]
\label{rem:oracle_supermodularity}
The monotone and supermodular properties of $f$ hold with respect to the oracle-induced flow realization of Definition~\ref{def:oracle_flow}. This dependence is intrinsic to cyber-physical optimization problems in which objective values are well-defined only for globally consistent completions, precisely the setting addressed by the conditioned sequential greedy framework.
\end{rem}

\end{document}